\documentclass{amsart}
\usepackage{amsthm,amsmath,amsfonts,amssymb,amscd,mathrsfs,graphics,diagrams, longtable}
\usepackage{txfonts,pxfonts}
\usepackage{hyperref,supertabular}
\usepackage{latexsym}
\usepackage{pstricks}
\usepackage{graphicx}
\usepackage{curves}

\newtheorem{thm}{Theorem}[section]
\newtheorem{lm}[thm]{Lemma}

\theoremstyle{definition}
\newtheorem{rem}[thm]{Remark}

\theoremstyle{remark}

\newcommand{\mb}{\mathbb}

\renewcommand{\mod}{\operatorname{
mod}}

\newcommand{\Z}{{\mb{Z}}}
\newcommand{\Q}{\mb{Q}}
\newcommand{\C}{\mb{C}}
\newcommand{\pone}{\mb{P}^1}
\renewcommand{\mp}{\mathfrak{p}}
\newcommand{\spec}{\operatorname{Spec}}

\begin{document}
\title[Integral Models]{Integral Models of Extremal Rational Elliptic Surfaces}
\author[Jarvis, Lang, and Ricks]{Tyler J. Jarvis, William E. Lang, and Jeremy R. Ricks}
\dedicatory{Dedicated to the memory of Jeremy R. Ricks and Steven Galovich}
\date{\today}
\maketitle
 
\begin{abstract}
Miranda and Persson classified all extremal rational elliptic surfaces in characteristic zero.  We show that each surface in Miranda and Persson's classification has an integral model with good reduction everywhere (except for those of type $X_{11}(j)$, which is an exceptional case),  and that every extremal rational elliptic surface over an algebraically closed field of characteristic $p > 0$ can be obtained by reducing one of these integral models mod $p$.

\end{abstract} 
 
\section*{Introduction}

An extremal rational elliptic surface is a smooth projective rational elliptic surface $X$ over an algebraically closed field $k$ together with a morphism $f: X \to \pone$ with a section, such that the Mordell-Weil group of the geometric general fibre of $f$ is finite.  Such surfaces were classified by R. Miranda and U. Persson over the complex numbers \cite{mp} and by W. E. Lang over algebraically closed fields of characteristic $p > 0$ in \cite{l1} and \cite{l2}. A summary of these classifications is provided in the appendix.  During his work on \cite{l2}, the senior author became convinced that the extremal rational elliptic surfaces in characteristic $p$ were closely linked to their characteristic zero counterparts.  For instance, there is no surface $X_{3333}$ with four singular fibres of type $I_3$ in characteristic three, but one might guess that the surface of type III in characteristic three comes from reducing $X_{3333}$ $\mod 3$ in such a way that three of the singular fibres come together.   There are several such situations throughout the classification, but at the time, the senior author was unable to formulate a precise theorem covering all cases.

Later, the senior author began to wonder if the characteristic zero extremal rational elliptic surfaces had integral models with good reduction everywhere.  Since a rational elliptic surface is the projective plane with nine points (possibly infinitely near) blown up, this author guessed that such models might exist, and realized that if enough such models could be produced, they would provide the link between characteristic zero and characteristic $p$ discussed in the last paragraph.

This paper completes the program outlined above.  We show that each surface found by Miranda and Persson has an integral model with good reduction everywhere (except for those of type $X_{11}(j)$, which is an exceptional case),  and that every extremal rational elliptic surface over an algebraically closed field of characteristic $p > 0$ can be obtained by reducing one of these integral models mod $p$.  

An interesting feature of this theory is that the integral models of the Miranda-Persson surfaces are not unique, and in fact, we need to use more than one integral model of some of the Miranda-Persson surfaces to get a full set of extremal rational elliptic surfaces in characteristic $p$.  Note that while the Miranda-Persson models are defined over $\Z$, when they are normalized using the conventions in \cite{ta}, the discriminants tend to be divisible by 2 and 3.  When this happens the reduction is no longer an elliptic surface.  Much of the work in this paper deals with getting these powers of 2 and 3 out of the discriminant.  Also, note that the Miranda-Persson classification works in all characteristics except 2, 3, and 5, so that this paper is mostly about the difficulties involving these small primes.

We should emphasize that the models produced here are integral Weierstrass models.  We know by Artin's theory of simultaneous resolution of rational double points \cite{a} that these models can be transformed into integral smooth models, but this requires extension of the base ring, and we do not know how to control how big an extension is needed.

We have confined our investigation to showing the existence of integral models and finding enough integral models to get all the surfaces in characteristic $p$ found in \cite{l1} and \cite{l2}.  We have not tried to study extremal rational elliptic surfaces over non-algebraically closed fields, so that our lifting results are not as strong as those available for elliptic curves.

It would be interesting to try to relate the results found here to the universal elliptic curves studied by Katz and Mazur in \cite{km}.  We have not attempted to do this here.  It would also be interesting to look at the implications of the existence of different integral models for the same surface for moduli of rational elliptic surfaces, but we have no ideas about this at this time. 

Here is a plan of the paper.  In Section~\ref{sec:zero}, we review some results on elliptic curves and surfaces that we will need in the rest of the paper, and set terminology.  In Section~\ref{sec:one}, we discuss integral models of extremal rational elliptic surfaces over the ring of integers $\Z$.  We find that three types of surfaces do not have integral models over $\Z$, and we find integral models for these over rings of integers in number fields in Section~\ref{sec:two}.  In Section~\ref{sec:three}, we look at the reduction $\mod p$ of the models found in Sections~\ref{sec:one} and \ref{sec:two}, and show that these reductions give us a full set of extremal rational elliptic surfaces over an algebraically closed field of characteristic $p$.

\setcounter{section}{-1}
\section{Preliminaries and terminology}\label{sec:zero}

An elliptic surface over a field $k$ is a smooth projective surface $X$ over $k$ together with a morphism $f: X \to C$, where $C$ us a smooth projective curve over $k$, 
%%TJ added
and such that all but finitely many fibres are elliptic curves. 
We will always assume $f$ has a fixed section (which we call \emph{zero}) and that there are no exceptional curves of the first kind in any fibre of $f$.

In this paper, we are studying rational elliptic surfaces, which implies $C = {\pone}$. Our surfaces are also extremal, which means that the Mordell-Weil group of the generic fibre is finite.  A complete classification of these surfaces over the complex field $\C$ has been given by Miranda and Persson \cite{mp}, and we will follow the notations of that paper for the most part.  

As is well known, an elliptic surface over ${\pone}$ has a Weierstrass model, which is obtained from the smooth model by contracting those components of each singular fibre which do not meet the zero section.  The resulting contracted surface has at most rational double points as singularities.  The Weierstrass model is birationally isomorphic to the surface defined by the Weierstrass equation
 $$y^2 + a_1xy + a_3y = x^3 + a_2x^2 +a_4x +a_6,$$
where the $a_i$ are polynomials in $t$, an affine coordinate on the base ${\pone}$.  There is a global minimal Weierstrass equation for our surface (see \cite{sil}1 , Ch. VII, Sect. 1, and Ch. VIII, Sect. 8, for discussion of the minimal Weierstrass equation), and the requirement that our surface is rational forces degree $a_i \le i$ for this minimal equation.

We are interested in finding integral models of extremal rational elliptic surfaces. This means we want diagrams

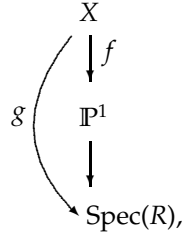
\setlength{\unitlength}{1mm}
\begin{picture}(50,42)(-10,0)
 \curve(43,8, 38,19, 43,31)
 \put(43,8){\vector(1,-1){1}}
 \put(35,20){$g$}
 \put(44,33){$X$}
 \put(45.5,31){\vector(0,-1){6}}
 \put(47,28){$f$}
 \put(44,19){${\pone}$}
 \put(45.5,17){\vector(0,-1){6}}
 \put(47,14){}
 \put(45,6){$\spec(R)$,}
\end{picture}

\noindent where $f$ and $g$ are flat and proper, $R$ is an integral domain (almost always the ring of integers of an algebraic number field), and such that the geometric general fibre of $g$ is the Weierstrass model of an extremal rational elliptic surface.  This setup will be called an \emph{integral model in the weak sense}.
%%TJ deleted
% of the geometric general fibre of $g$}.  
An \emph{integral model in the strong sense} will satisfy the extra condition that it will have good reduction at all primes of $R$, meaning that all fibres of $g$ will be Weierstrass models of elliptic surfaces, and that all the Weierstrass equations will be minimal.  If we use the term integral model without qualification, we will mean integral model in the strong sense.

Our strategy will be to start with the Miranda-Persson models for each type of extremal rational elliptic surface, which are integral models in the weak sense, and  modify them so they become integral in the strong sense.  Therefore we will review briefly how the Weierstrass equations transform under various substitutions.  Here we follow Tate's ``formulaire'' \cite{ta}.

If we have two Weierstrass equations over $R$,
 \begin{align*}
y^2 &+ a_1xy + a_3y = x^3 + a_2x^2 +a_4x +a_6 \text{ and}\\
y^{\prime 2} &+ a_1'x'y' + a_3' y' = x^{\prime 3} + a_2' x^{\prime 2} +a_4' x' +a_6'
 \end{align*}
and an isomorphism of the underlying elliptic curves preserving the origin, then this isomorphism is obtained by a substitution of the form 
 \begin{align*}
x = u^2x' + r\\
y = u^3y' +su^2x' +q,
 \end{align*}
where $u$ is a unit of $R$, and $r$, $s$, and $q$ are elements of $R$.  (We use $q$ instead of Tate's $t$ since we are using $t$ as a coordinate on ${\pone}$.)  The coefficients of the equations are related by
 \begin{align*}
ua_1' &= a_1 + 2s\\
u^2a_2' &= a_2 - sa_1 +3r - s^2\\
u^3a_3' &= a_3 +ra_1 +2q\\
u^4a_4' &= a_4 - sa_3 +2ra_2 - (q + rs)a_1 +3r^2 - 2st\\
u^6a_6' &= a_6 + ra_4 + r^2a_2 + r^3 - qa_3 - q^2 - rqa_1.
 \end{align*}
 
In practice, we will restrict ourselves to substitutions of two types:
\begin{enumerate}
\item[a.)]  Those where $u = 1$, which will be called \emph{$rsq$ substitutions}, and 
\item[b.)]  Those where $r=s=q=0$. These will be called \emph{$u$-substitutions}.
\end{enumerate}

We will be using the auxiliary quantities
 \begin{align*}
b_2 &= a_1^2 + 4a_2\\
b_4 &= a_1a_3 + 2a_4\\
b_6 &=a_3^2 + 4a_6\\
b_8 &= a_1^2a_6 - a_1a_3a_4 + 4a_2a_6 + a_2a_3^2 - a_4^2\\
c_4 &= b_2^2 - 24b_4\\
c_6 &= -b_2^3 + 36b_2b_4 - 216b_6\\
\Delta &= b_2^2b_8 - 8b_4^3 - 27b_6^2 + 9b_2b_4b_6\\
j &= c_4^3/\Delta.
 \end{align*}
 
We know that our Weierstrass equation defines an elliptic curve if and only if $\Delta \ne 0$.

\section{Integral models over $\Z$ 
}\label{sec:one}
\renewcommand{\thesubsection}{\S\thesection.\Alph{subsection}}
\subsection{$X_{11}(j)$}\label{sec:onea}
\ 

This surface is different from all the others in several ways, so we deal with it first.  Here we have a family of surfaces, one for each value of $j$.  The Miranda-Persson model has Weierstrass equation 
 $$y^2 = x^3 +rt^2x + st^3,$$
where $r$ and $s$ are chosen so that $j = 6912r^3/(4r^3 +27s^2)$.  (Note that we use $t$ as an affine coordinate on ${\pone}$, while Miranda-Persson use $u$ and $v$ as projective coordinates.  Also our $j$ differs from theirs by a factor of 1728 and our $\Delta$ differs from theirs by a factor of 16).  We observe also that this surface is a quadratic twist of the constant elliptic curve defined by $y^2 = x^3 + rx + s$.  Since there is no constant elliptic curve over $\Z$ with good reduction everywhere, we do not try to find an integral model in the strong sense for $X_{11}(j)$.  Instead, we construct our model as a quadratic twist of a constant elliptic curve in such a way that our model has good reduction everywhere that the original curve did.         

We begin with an elliptic curve $E$ over a ring $R$ where 2 is not a zero-divisor.  We let $E$ have Weierstrass equation $y^2 + g_1xy + g_3y = x^3 + g_2x^2 +g_4x +g_6$. 

Now we perform an rsq-substitution with $r = 0$, $s = -(1/2)g_1$, and $q = -(1/2)g_3$.  This produces an elliptic curve over $R[1/2]$ with equation $y^2  = x^3 + (g_2+(1/4)g_1^2)x^2 +(g_4 +(1/2)g_1g_3)x + g_6 +(1/4)g_3^2$.  Now we perform a quadratic twist over $R[1/2, (t^2 +4t)^{1/2}]$ and obtain an elliptic surface with equation $y^2 = x^3 + (g_2+(1/4)g_1^2)(t^2 +4t)x^2 +(g_4 +(1/2)g_1g_3) (t^2 +4t)^2x +(g_6 +(1/4)g_3^2) (t^2 +4t)^3$.

Finally, we perform an rsq-substitution with $r =0$, $s = (1/2)g_1t$, and $q = (1/2)g_3t^3$ to get the surface with equation
 \begin{align*}
y^2 &+ g_1txy + g_3t^3 y = x^3 + (g_2t^2 +4g_2t +g_1^2)x^2\\
&+(g_4t^4 +8g_4t^3 +16g_4t^2 +4g_1g_3t^3 + 8g_1g_3t^2)x\\
&+ (g_6t^6 +12g_6t^5 + 48g_6t^4 +64g_6t^3 +3g_3^3t^5 + 12g_3^2t^4 + 16g_3^2t^3).
 \end{align*}

This is the desired model.  Its discriminant is $\Delta_E(t^2 +4t)^6$, where $\Delta_E$ is the discriminant of the original elliptic curve $E$, and its $j$-invariant is the same as the $j$-invariant of $E$.  Our surface is an integral model in the weak sense, and it has good reduction at all points of $\spec(R)$ where the original constant curve has good reduction.

Note that we could not twist over $R[1/2, \sqrt t]$ as in \cite{mp}, for then we would have obtained a model where the singular fibres are at 0 and $\infty$.  This model would have reduced $\mod 2$ to a surface with two singular fibres of a sort which is not permitted by \cite{l2}.
 \bigskip

\subsection{$X_{22}$, $X_{33}$, and $X_{44}$}
\ 

We will show that these surfaces do not have integral models over ${\Z}$.  Note that $X_{22}$ and $X_{44}$ have constant $j$-invariant 0, while $X_{33}$ has constant $j$-invariant $12^3$.  The non-existence of integral models over ${\Z}$ for these surfaces will follow easily from the following lemma, which is well known to the experts (at least in the number field case).

 \begin{lm}\label{lem1}
Suppose E is an elliptic curve over a discrete valuation ring R, and suppose $\operatorname{Fract}(R)$ has characteristic 0. 
 \begin{enumerate}
\item[a)] Suppose 3 is a uniformizer of $R$, and suppose $j(E) = 0$.  Then $E$ has bad reduction.

\item[b)] Suppose 2 is a uniformizer of $R$, and suppose $j(E) = 12^3$.  Then $E$ has bad reduction.
 \end{enumerate}
 \end{lm}
\begin{proof}
 a.) Suppose $E$ has good reduction.  Since $j(E) = 0$, and $j =c_4^3/\Delta$, we get $c_4 = 0$.  Since $1728\Delta = c_4^3 - c_6^2$ and $v_3(1728\Delta) = v(1728) =3$, while $v_3(c_6^2)$ is even, we have a contradiction.

b.) Since $1728\Delta = c_4^3 - c_6^2$ and $j = c_4^3/\Delta$, the fact that $j =1728$ implies $c_6 = 0$. If we have good reduction at 2, then 4 exactly divides $c_4$.  Since $c_4 = b_2^2 - 24b_4$, this means 2 exactly divides $b_2$.  But $b_2 = a_1^2 + 4a_2$, which gives a contradiction. Now an integral model over ${\Z}$ can be considered as an elliptic curve over ${\Z}[t]$ with good reduction at all prime ideals generated by rational primes.  Applying the lemma to the cases where $R$ equals ${\Z}[t]$ localized at (2) or (3), we see integral models of the three surfaces given cannot exist. 
\end{proof}

We will see in Section~\ref{sec:two} that these surfaces have integral models over rings of integers in number fields.

\subsection{$X_{3333}$}
\

While it would not be useful to show all the details of the way we find the integral model, a sketch of the method may be useful. Some of the other integral models were found using a similar technique.

The Miranda-Persson model of $X_{3333}$ has Weierstrass equation $y^2  = x^3 + (-3t^4 + 24t)x + (2t^6 + 40t^3 -16)$.  The discriminant of this is $\Delta = 2^{12}3^3(t^3 + 1)^3$.  This is an integral model in the weak sense, and to produce an integral model in the strong sense,  we must eliminate the factors of 2 and 3 which appear in $\Delta$.

For this, we use Tate's algorithm at the prime ideals (2) and (3) of ${\Z}[t]$.  Roughly speaking, the algorithm begins by making a series of $rsq$-substitutions.  When these are done, either the algorithm stops, or a $u$-substitution is made, which has the effect of dividing the discriminant by $u^{12}$.  The above steps are then repeated.  In order to get rid of the unwanted factors, it is natural to start by making the discriminant divisible by $3^{12}$.  This is done by replacing $t$ by $9t + 8$.

After making this change, we ran Tate's algorithm separately at the primes (2) and (3).  At the prime (3), the algorithm went through to the end, so that the factor $3^{12}$ could be eliminated.  At the prime (2), the algorithm terminated before the end.  Thus, the minimal version of our modified Miranda-Persson model had good reduction at 3, but bad reduction at 2.  To deal with this, we extended scalars to the Gaussian integers, and ran the algorithm again.  We obtained a new surface whose Weierstrass equation had integer coefficients, had good reduction at 2, and was isomorphic to the Miranda-Persson model over ${\Q}(i)$.

Finally, we used the Chinese remainder theorem to get substitutions that achieved the desired goals simultaneously at 2 and 3.  The final integral model is $y^2 + a_1xy + a_3y = x^3 + a_2x^2 + a_4x + a_6$, with the $a_i$ defined below.
 \begin{align*}
a_1 &= 171t\\
a_2 &= 16-7353t\\
a_3 &= -3\\
a_4 &= 594t^4 -54t^3 -528t^2 +214t +76\\
a_6 &= -2700t^6 +648t^5 + 3924t^4 - 3t^3 - 1682t^2 - 304t +88
 \end{align*}
The discriminant is $\Delta = -(t + 1)^3(27t^2 +45t +19)^3$.
 \bigskip
 
\subsection{$X_{222}$}
\  

This is the only model found by a computer search, which was carried out with the assistance of Jason Grout. The singular fibres of this surface are of types $I_2^*$, $I_2$, and $I_2$.  So we begin by writing down a Weierstrass equation for a surface with a singular fibre of type $I_2^*$ at $t = 0$.  This equation has coefficients 
 \begin{align*}
a_1 &= tc_0\\
a_2 &= t(a + bt)\\
a_3 &= t^3d_0\\
a_4 &= t^3(ct + d)\\
a_6 &= t^5(et + f). 
 \end{align*}
We then let the computer vary $c_0$, $d_0$, $a$, $b$, $c$, $d$, $e$, and $f$, until it found a surface whose discriminant was $t^8$ times a perfect square.  Again, we are grateful to Jason Grout for help with the computer programming and with selecting the search region.  An integral model yielded by this search has Weierstrass coefficients 
 \begin{align*}
a_1 &= t\\
a_2 &= t + t^2\\
a_3 &=4t^3\\
a_4 &= 2t^4\\
a_6 &= 4t^5 + t^6.
 \end{align*}
The discriminant is $\Delta = -t^8(16 +40t +89t^2)^2$. 
 \bigskip
 
\subsection{$X_{321}$ }
\ 

We found two integral models for this surface.  The first integral model we call 321A, which has Weierstrass coefficients
 \begin{align*}
a_1 &= 1\\
a_2 &= -1\\
a_3 &= 6t + 4\\
a_4 &= -4t- 2\\
a_6 &= -9t^2 -12t -4\\
\Delta &= t^2(64t + 9).
 \end{align*}

A second integral model of this surface will be needed in Section~\ref{sec:three}.  This integral model we call 321$B$, with Weierstrass coefficients
 \begin{align*}
a_1 &= t\\
a_2 &= 0\\
a_3 &= 0\\
a_4 &= t^3\\
a_6 &= 0\\
\Delta &= t^9(t - 64).
 \end{align*}
(This is the simplest possible surface with a fibre of type III* at $t = 0$.) 

Notice that these surfaces are not isomorphic over ${\Z}$.  The reduction of model $321A \mod 3$ has two singular fibres, while the reduction of model $321B \mod 3$ has three singular fibres.  This example shows that integral models of extremal rational elliptic surfaces are not unique.

In the remaining cases we simply give the Weierstrass coefficients of the integral models.  In some of the cases, we will need two models of the same surface.
 \bigskip

\subsection{$X_{9111}$ }
\

 \begin{tabbing}
\hskip 1.25in \= $a_1 = t$\\
\> $a_2 = 0$\\
\> $a_3 = -1$\\
\> $a_4 = 0$\\
\> $a_6 = 0$\\
\> $\Delta = -(t + 3)(t^2 - 3t + 9)$.
 \end{tabbing}

\subsection{ $X_{5511}$}
\ 

 \begin{tabbing}
\hskip 1.25in \= $a_1 = 5t +1$\\
\> $a_2 = -6t^2 - 4t - 3$\\
\> $a_3 = 1$\\
\> $a_4 = 2$\\
\> $a_6 = - t - 1$\\
\> $\Delta = t^5(t^2 - 11t - 1)$.
 \end{tabbing}

\subsection{$X_{8211}$}
\ 
 
 \begin{tabbing}
\hskip 1.25in \= Model 8211A\\
\> $a_1 = 1$\\
\> $a_2 = 32t^2$\\
\> $a_3 = 0$\\
\> $a_4 = 256t^4$\\
\> $a_6 = -t^6-64t^4$\\
\> $\Delta = t^2(1 + 16t^2)$.\\[4mm]
\> Model 8211B\\
\> $a_1 = t$\\
\> $a_2 = 128$\\
\> $a_3 = 0$\\
\> $a_4 = 21t^2 + 5461$\\
\> $a_6 = 441t^2 + 77568$\\
\> $\Delta = t^2(t^2 + 16)$.
 \end{tabbing}

\subsection{$X_{6321}$ }
\ 

 \begin{tabbing}
\hskip 1.25in \= Model 6321A\\
\> $a_1 = 1$\\
\> $a_2 = 4t^2 + 2t$\\
\> $a_3 = t$\\
\> $a_4 = 2t^6$\\
\> $a_6 = 0$\\
\> $\Delta = t^3(t +1)(-1 + 8t)^2$.\\[4mm]
\> Model 6321B\\
\> $a_1 = t$\\
\> $a_2 = 1 + t$\\
\> $a_3 = 2t^2 + t$\\
\> $a_4 = t - t^3$\\
\> $a_6 = -t^3 - t^4$\\
\> $\Delta = (t + 8)(t -1)^2t^3$.
 \end{tabbing}

\subsection{$X_{4422}$}
\ 

 \begin{tabbing}
\hskip 1.25in \= $a_1 = 1$\\
\> $a_2 = 4t^2$\\
\> $a_3 = 4t^2$\\
\> $a_4 = -t^2$\\
\> $a_6 = -4t^4$\\
\> $\Delta = t^4(4t - 1)^2(4t + 1)^2$.
 \end{tabbing}

\subsection{$X_{211}$}
\ 

 \begin{tabbing}
\hskip 1.25in \= $a_1 = 1$\\
\> $a_2 = 0$\\
\> $a_3 = 0$\\
\> $a_4 = -72t(432t + 1)(373248t^2 + 864t + 1)$\\
\> $a_6 = 557256278016t^5 + 2257403904t^4 + 2985984t^3 + 864t^2 - t$\\
\> $\Delta = t(432t + 1)(864t+1)^{10}$.
 \end{tabbing}

\subsection{$X_{431}$}
\ 

 \begin{tabbing}
\hskip 1.25in \= $a_1 = 1$\\
\> $a_2 = -27t$\\
\> $a_3 = 0$\\
\> $a_4 = 243t^2$\\
\> $a_6 = -729t^3 - 27t^2 + t$\\
\> $\Delta = -t(27t +1)^3$.
 \end{tabbing}

\subsection{$X_{411}$}
\ 

 \begin{tabbing}
\hskip 1.25in \= $a_1 = 1$\\
\> $a_2 = 32t +3$\\
\> $a_3 = 3$\\
\> $a_4 = 256t^2 + 64t + 2$\\
\> $a_6 = 192t^2 + 31t -1$\\
\> $\Delta = t(16t + 1)$.
 \end{tabbing}

\subsection{$X_{141}$}
\ 

 \begin{tabbing}
\hskip 1.25in \= $a_1 = 1$\\
\> $a_2 = -256t^2 + 24t$\\
\> $a_3 = 10t^2 + 20t$\\
\> $a_4 = -117t^2$\\
\> $a_6 = -25t^4 -100t^3 - 112t^2 +t$\\
\> $\Delta = t(16t - 1)^7$.
 \end{tabbing}

\section{Integral models over larger rings}\label{sec:two}

In this section, we find integral models over rings of integers of algebraic number fields of the three surfaces which do not have integral models over ${\Z}$.
 \bigskip

\subsection{$X_{33}$}
\

This surface has $j$-invariant $12^3$ and two singular fibres of types III and III*.  Our approach to finding an integral model was to start with a constant elliptic curve over ${\Z}[i]$ with $j$-invariant $12^3$ and with good reduction at $(1+i)$, the unique prime ideal of ${\Z}[i]$ lying over 2.  We then replaced this with an elliptic surface with singular fibres of the desired types which is $(1+i)$-adically close to our constant curve.  It was necessary choose the new surface to be sufficiently close $(1+i)$-adically to the constant curve to have good reduction at $(1+i)$, but not so close as to have constant reduction (that is, reduction not depending on $t$) at $(1+i)$. 

Our starting curve was defined by $y^2 = x^3 +(-1 + 2i)x$.  (This is an interesting curve over ${\Z}[i]$, since its minimal model only has bad reduction at one Gaussian prime, namely $(-1 +2i)$.  We replaced this by the surface with Weierstrass equation $y^2 = x^3 +(8t - 1 + 2i)x$.  After running Tate's algorithm at $(1+i)$, we found an integral model in the strong sense over ${\Z}[i]$ of $X_{33}$.  Its Weierstrass coefficients are
 \begin{align*}
a_1 &= 1 - i\\
a_2 &= -i\\
a_3 &= -i\\
a_4 &= -2t\\
a_6 &= it\\
\Delta &= (8t - 1 + 2i)^3.
 \end{align*}

We found integral models in the remaining cases over ${\Z}[3^{1/4}]$.  According to the computer package \emph{pari}, this is the full ring of integers in $\Q(3^{1/4})$, and this ring has class number one.  We omit the method of construction of the models, which was similar to the previous case.
 \bigskip

\subsection{$X_{22}$}
\ 

 \begin{align*}
a_1 &= 0\\
a_2 &= -16\sqrt 3\\
a_3 &= 27\cdot3^{1/4}\\
a_4 &= 256\\
a_6 &= t - 637\\
\Delta &= -169 - 312t\sqrt 3 - 432t^2.
 \end{align*}

\subsection{$X_{44}$}
\ 

 \begin{align*}
a_1 &= 0\\
a_2 &= -16\sqrt 3\\
a_3 &= 27\cdot3^{1/4}\\
a_4 &= 256\\
a_6 &= -376\sqrt 3 +97t + 3\sqrt 3t^2\\
\Delta &= -(1/9)(18t +97\sqrt 3)^4.
 \end{align*}

Note that the constant term of $\Delta$ is not divisible by $3^{1/4}$ in ${\Z}[3^{1/4}]$.

\section{Reductions}\label{sec:three}

In this section, we find the reductions modulo primes of the surfaces found in the first two sections.  The list of characteristic $p$ surfaces produced will be sufficient to show that all extremal rational elliptic surfaces in characteristic $p$ come from characteristic zero.  Proofs will not be given, since checking that the reduced surfaces are of the stated types is an easy exercise for the reader.  However, we will prove the following lemma, which may speed the task of verifying the results.

 \begin{lm}
Suppose $X \to \pone_R \to {\rm Spec}\, R$ is a minimal Weierstrass model of an extremal rational elliptic surface, where $R$ is the ring of integers of an algebraic number field.  Suppose $\mp$ is a prime of $R$ and suppose $X$ has good reduction at $\mp$, which means that the geometric general fibre $X_\mp \to R/\mp \equiv k$ is the minimal Weierstrass model of an elliptic surface.  Then $X_\mp$ is also an extremal rational elliptic surface.
 \end{lm}

\begin{proof} By enlarging $R$ to a bigger ring $S$ and using Artin's theory \cite{a} of simultaneous resolution of rational double points, we get a smooth model $Y$ of $X$ such that $Y \to {\rm Spec}\, S$ is smooth.  By enlarging $S$ further if necessary, we may assume that all components of singular fibres are rational over $S$.  Since $Y$ is extremal, the sublattice of $NS(Y)$ (the Neron-Severi group) generated by those components of singular fibres not meeting the zero solution is of rank 8.  Therefore the sublattice of $NS(Y_{\mp})$ generated by components of fibres not meeting the zero solution also has rank 8 and we see $Y_{\mp}$ (and hence also $X_{\mp}$) is extremal.
\end{proof}
A slight extension of the above argument gives
 \begin{lm}
The order of the Mordell-Weil group of $X_{\mp}$ divides the order of the Mordell-Weil group.  (Here we use Mordell-Weil group to mean the order of the Mordell-Weil of the general fibre of $X \to \pone_K$, where $K$ is an algebraically closed field.)
 \end{lm}

 \begin{lm}
Let $p$ be a prime, $p \ne 2, 3$.  Then the reduction $\mod p$ of each model of an extremal rational elliptic surface found above is an extremal rational elliptic surface of the same type, with one exception.  The reduction at $p = 5$ of the given model of $X_{5511}$ is the unique extremal rational elliptic surface in characteristic 5 with singular fibres of types II, I$_5$, and I$_5$.
 \end{lm}

The existence of this characteristic 5 surface was pointed out by Chad Schoen, who should have been thanked in \cite{l2}.
We will call this surface the \emph{Schoen Surface}.

Now we give a table of the reduction behavior of the surfaces found in the previous sections at characteristics 2 and 3.  For the surfaces not defined over ${\Z}$, reduction mod 2 or 3 means reduction at primes lying over these rational primes.  The numbers labeling the $\mod 2$ and $\mod 3$ reductions come from \cite{l2}.  The explanation of the reduction behavior of surfaces of type $X_{11}(j)$ comes after the table.

\begin{longtable}[c]{|lll|}
\caption{Reduction of Integral Models mod $2$ and $3$}
      \\
      \hline
Integral    & Reduction & Reduction \\
Model      & mod $2$ & mod $3$ \\
\hline
\hline
\endfirsthead
        \hline
        \multicolumn{3}{l}{\small\slshape continued from previous page}\\
        \hline
Integral    & Reduction & Reduction \\
Model      & mod $2$ & mod $3$ \\
    \hline
\endhead
        \hline
        \multicolumn{3}{r}{\small\slshape continued on next  page}\\
        \hline            \endfoot
\hline
 \endlastfoot
$X_{11}(j)$  & I  & VI\\
  & II & VI bis\\
$X_{33}$& II & V\\
$X_{22}$& II & I\\
$X_{44}$& VII & I\\
$X_{3333}$ & $X_{3333}$ & III\\
$X_{222}$& I & IX\\
$X_{321A}$ & V & V\\
$X_{321B}$& V & XI\\
$X_{9111}$& $X_{9111}$ & II\\
$X_{5511}$& $X_{5511}$ & $X_{5511}$\\
$X_{8211A}$ & V & $X_{8211}$\\
$X_{8211B}$& III & $X_{8211}$\\
$X_{6321A}$& IX & VII\\
$X_{6321B}$& VIII & VII\\
$X_{4422}$& IV & $X_{4422}$\\
$X_{211}$& VI & IV\\
$X_{431}$& IX & IV\\
$X_{411}$& VI & X\\
$X_{141}$& VI & VIII
\end{longtable}

 \begin{thm}
All extremal rational elliptic surfaces over algebraically closed fields of characteristic $p$ lift to characteristic zero.
 \end{thm}  

\begin{proof}  The only cases which are not obvious from the above table are those related to $X_{11}(j)$.  Suppose we have a surface of type $X_{11}(j)$ in characteristic $p$, $p \ne 2, 3$.  Choose an elliptic curve $E$ over some ring $R$ with good reduction at some prime ideal ${\mp}$ containing $p$, and such that $j(E) \equiv j \pmod {\mp}$.  Then perform a quadratic twist as in Section~\ref{sec:onea}.  We get an elliptic surface over $R$ with good reduction at ${\mp}$ of type $X_{11}(j)$.

Further explanation of the reduction behavior of $X_{11}(j)$ at 2 and 3 is needed.  For characteristic 2, if we produce a model of $X_{11}(j)$ as in Section~\ref{sec:onea}, where $E/R$ is chosen to have good reduction at a prime ${\mp}$ containing 2, the reduction of our model at ${\mp}$ will have discriminant $\Delta_Et^{12}$. Thus, it will be of Type I or Type II.  The Weierstrass equation of a surface of Type I is $y^2 + txy + a_3y = x^3 + tx^2 +kt^6$, $k \ne 0$, and the Weierstrass equation of Type II is $y^2  + t^3y= x^3 + t^5$.  The $j$-invariant of the general fibre of type I is $1/k$.  So if we choose such that $j(E) \equiv 1/k$, then the reduction of our model of $X_{11}(j)$ will be isomorphic (over the algebraic closure of $R/{\mp}$) to the surface of type I with $k$ as specified.  The $j$-invariant of the general fibre of the surface of type II is 0.  So if we choose $E$ so that $j(E) \equiv 0 \pmod {\mp}$, we get a surface of type II.  So all surfaces of type I or type II are liftable.  Incidentally the above discussion shows that type II in characteristic 2 is a specialization of type I, which may not have been obvious from \cite{l2}.

The situation in characteristic 3 is similar.  If we have a model of $X_{11}(j)$ as in Section~\ref{sec:onea} (we might as well choose $g_1 = g_3 = 0$ for simplicity), its reduction at a prime containing 3 will be of type VI or VI bis.  The equations for surfaces of types VI and VI bis are $y^2  = x^3 + tx^2 + kt^3$ $(k \ne 0)$ and $y^2 = x^3 + t^2x$ respectively.  The $j$-invariants are $-1/k$ and 0 respectively.  So if we choose $E/R$ and the prime ${\mp}$ of $R$ containing 3 such that $j(E) \equiv -1/k \pmod {\mp}$, then the reduction of our model of $X_{11}(j)$  will be isomorphic (over the algebraic closure of $R/{\mp}$) to a surface of type VI with $k$ as specified.  If our setup is such that $j(E) \equiv 0 \pmod {\mp}$, the reduction will be of type VI bis.  So all surfaces of type VI or type VI bis are liftable.  The above discussion shows that type VI bis is a specialization of type VI.
\end{proof}

\setcounter{section}{0}
\renewcommand{\thesection}{\Alph{section}}
\renewcommand{\thesubsection}{\thesection.\arabic{subsection}}
\section{Appendix}

In this appendix we briefly summarize the results of 
\cite{mp} and \cite{l2}.

\subsection{Characteristic $0$ and Charachteristic $\ge 5$}

In \cite{mp} Miranda and Persson prove that over $\C$ every extremal rational elliptic surface is one of the following, and for each of the following configurations of singular fibers except $I_0^*\,I_0^*$, and $II\,I_5\,I_5$ there is a unique surface with that configuration.
For the configuration $I_0^*\,I_0^*$ there is precisely one surface $X_{11}(j)$ for each $j\in \C$, and the configuration 
$II\,I_5\,I_5$ does not occur in characteristic $0$.

%\begin{tabular}{lr}
%& Singular\\ 
%Surface & fibers\\
%\hline\hline
%$X_{22}$ & $II\,II^*$ \\
%$X_{33}$ & $III\,III^*$ \\
%$X_{44}$ & $IV\,IV^*$ \\
%$X_{11}(j)$ $j\in\C$ & $I_0^*\,I_0^*$ \\
%$X_{211}$ & $II^*\,I_1\,I_1$ \\
%$X_{321}$ & $III^*\,I_2\,I_1$ \\
%$X_{431}$ & $IV^*\,I_3\,I_1$ \\
%$X_{141}$ & $I_1^*\,I_4\,I_1$ \\
%$X_{222}$ & $I_2^*\,I_2\,I_2$ \\
%$X_{9111}$ & $I_9\,I_1\,I_1\,I_1$ \\
%$X_{8211}$ & $I_8\,I_2\,I_1\,I_1$ \\
%$X_{5511}$ & $I_5\,I_5\,I_1\,I_1$ \\
%$X_{6321}$ & $I_6\,I_3\,I_2\,I_1$ \\
%$X_{4422}$ & $I_4\,I_4\,I_2\,I_2$ \\
%$X_{3333}$ & $I_3\,I_3\,I_3\,I_3$ 
%\end{tabular}

\setcounter{table}{0}
\renewcommand{\thetable}{\thesection.\arabic{table}}

\begin{longtable}[c]{|llcc|}
\caption{Extremal Rational Elliptic Surfaces in Char $\ge 5$}
      \\
      \hline
Surface   & Singular &Char $0$ and&\\
Name      & Fibers & Char $>5$ & Char $5$\\
\hline
\hline
\endfirsthead
        \hline
        \multicolumn{4}{l}{\small\slshape continued from previous page}\\
        \hline
Surface   & Singular & Char $0$ and & \\
Name      & Fibers &  Char $>5$ & Char $5$  \\
    \hline
\endhead
        \hline
        \multicolumn{4}{r}{\small\slshape continued on next  page}\\
        \hline            \endfoot
\hline
 \endlastfoot
$X_{22}$ & $II\,II^*$ & \checkmark & \checkmark \\
$X_{33}$ & $III\,III^*$ & \checkmark & \checkmark  \\
$X_{44}$ & $IV\,IV^*$ & \checkmark & \checkmark \\
$X_{11}(j)$ $j\in\C$ & $I_0^*\,I_0^*$ & \checkmark & \checkmark \\
$X_{211}$ & $II^*\,I_1\,I_1$ & \checkmark & \checkmark \\
$X_{321}$ & $III^*\,I_2\,I_1$ & \checkmark & \checkmark \\
$X_{431}$ & $IV^*\,I_3\,I_1$ & \checkmark & \checkmark \\
$X_{411}$ & $I_4^*\,I_1\,I_1$ & \checkmark & \checkmark \\
$X_{141}$ & $I_1^*\,I_4\,I_1$ & \checkmark & \checkmark \\
$X_{222}$ & $I_2^*\,I_2\,I_2$ & \checkmark & \checkmark \\
$X_{9111}$ & $I_9\,I_1\,I_1\,I_1$ & \checkmark & \checkmark \\
$X_{8211}$ & $I_8\,I_2\,I_1\,I_1$ & \checkmark & \checkmark \\
$X_{5511}$ & $I_5\,I_5\,I_1\,I_1$ & \checkmark & dne \\
   Schoen  & $II\, I_5\, I_5$  & dne & \checkmark \\
$X_{6321}$ & $I_6\,I_3\,I_2\,I_1$ & \checkmark & \checkmark \\
$X_{4422}$ & $I_4\,I_4\,I_2\,I_2$ & \checkmark & \checkmark \\
$X_{3333}$ & $I_3\,I_3\,I_3\,I_3$ & \checkmark & \checkmark \\
\end{longtable}

In \cite{l1,l2} it is shown that in characteristics not equal to $2$, $3$, or $5$ the results of \cite{mp} also hold.  In characteristic $5$ the results of \cite{mp} hold except that the surface $X_{5511}$ is replaced by a (unique) surface with singular fibers  $II\, I_5\, I_5$, which we call the \emph{Schoen surface}.

\subsection{Chararacteristics $2$ and $3$}

Characteristics $2$ and $3$ are quite different from the others.  

\subsubsection{Characteristic $3$}
In \cite{l1,l2} it is shown that in characteristic $3$ every extremal rational elliptic surface is one of the following, and for each of the following configurations of singular fibers except $I_0^*\,I_0^*$, there is a unique surface with that configuration.  For the configuration  $I_0^*\,I_0^*$, there is a family of surfaces $VI(k)$ parameterized by non-zero values of $k$ and one additional surface $VI$bis. 
\newpage
\begin{longtable}[c]{|ll|}
\caption{Extremal Rational Elliptic Surfaces in Char $3$}
      \\
      \hline
Surface   & Singular \\
Name      & Fibers \\
\hline
\hline
\endfirsthead
        \hline
        \multicolumn{2}{l}{\small\slshape continued from previous page}\\
        \hline
Surface   & Singular   \\
Name      & Fibers   \\
    \hline
\endhead
        \hline
        \multicolumn{2}{r}{\small\slshape continued on next  page}\\
        \hline            \endfoot
\hline
 \endlastfoot
$I$ & $II^*$  \\
$II$ & $II\, I_9$  \\
$III$ & $IV^*\,I_3$ \\
$IV$ & $II^*\, I_1$  \\
$V$ & $III^*\, III$  \\
$VI(k) $ $k\neq 0$& $I_0^*\, I_0^*$  \\
$VI$ bis & $I_0^*\, I_0^*$ \\
$VII$ & $III\, I_3\,I_6$ \\
$VIII$ & $I_1^*\,I_1\,I_4$ \\
$IX$ & $I_2^*\,I_2\,I_2$ \\
$X$ & $I_4^*\,I_1\,I_1$\\
$XI$ & $III^*\,I_1\,I_2$\\            
$X_{8211}$ & $I_8\,I_2\,I_1\,I_1$ \\
$X_{5511}$ & $I_5\,I_5\,I_1\,I_1$ \\
$X_{4422}$ & $I_4\,I_4\,I_2\,I_2$  \\
\end{longtable}

\subsubsection{Characteristic $2$}
Finally, in \cite{l1,l2} it is shown that in characteristic $2$ every extremal rational elliptic surface is one of the following, and for each of the following configurations of singular fibers except $I_4^*$, there is a unique surface with that configuration.  For the configuration  $I_4^*$, there is a family of surfaces $I(k)$ parameterized by non-zero values of $k$. 

\begin{rem}
Note that the the surfaces in characteristic $2$ do \emph{not} necessarily correspond to surfaces with the same name in characteristic $3$.
\end{rem}

\begin{longtable}[c]{|ll|}
\caption{Extremal Rational Elliptic Surfaces in Char $2$}
      \\
      \hline
Surface   & Singular \\
Name      & Fibers \\
\hline
\hline
\endfirsthead
        \hline
        \multicolumn{2}{l}{\small\slshape continued from previous page}\\
        \hline
Surface   & Singular   \\
Name      & Fibers   \\
    \hline
\endhead
        \hline
        \multicolumn{2}{r}{\small\slshape continued on next  page}\\
        \hline            \endfoot
\hline
 \endlastfoot
$I(k)$ $k\neq 0$ & $I_4^*$  \\
$II$ & $II^*$  \\
$III$ & $III\,I_8$ \\
$IV$ & $I_1^*\, I_4$  \\
$V$ & $III^*\, I_2$  \\
$VI $ & $II^*\, I_1$  \\
$VII$ & $IV\,IV^*$ \\
$VIII$ & $IV\,I_2\,I_6$ \\
$IX$ & $IV^*\,I_1\,I_3$ \\
$X_{9111}$ & $I_9\,I_1\,I_1\,I_1$  \\
$X_{5511}$ & $I_5\,I_5\,I_1\,I_1$  \\
$X_{3333}$ & $I_3\,I_3\,I_3\,I_3$ \\
\end{longtable}

\section*{Acknowledgments}

This paper began as a research project by Jeremy Ricks, a BYU undergraduate student, under the direction of Tyler Jarvis.  Unfortunately, Jeremy died in a tragic accident in September of 2001, before his work on the project was completed.  Both surviving authors were greatly impressed by Jeremy's mathematical ability and his dedication to this project.  He produced many of the models given here in a remarkably short amount of time.  We are extremely grateful to Jeremy's wife, Melinda, for providing us with Jeremy's notes.  Without them, this paper would not exist.  

We also wish to remember Professor Steven Galovich, who passed away in 2006.  Steve was the senior author's mentor at Carleton College (long before the concept of undergraduate mentoring became fashionable) and provided him with an outstanding introduction to the world of algebra and number theory.

The senior author would also like to thank C. Schoen, D. Doud, J. Grout, and K. Rubin for useful conversations on elliptic curves and surfaces, and Brigham Young University  for computer support.

Finally, we would like to thank M.~Sch\"utt and the referee for their comments on a previous version of this paper.


\begin{thebibliography}{MP}

\bibitem[A]{a} Artin, M.  \emph{Algebraic construction of Brieskorn's resolutions}. J. Algebra \textbf{29} (1974), 330--348.
 
\bibitem[KM]{km} 
Katz, N.;  Mazur, B. 
\emph{Arithmetic moduli of elliptic curves.} 
Annals of Mathematics Studies, \textbf{108}. Princeton University Press, Princeton, NJ, 1985.

\bibitem[L1]{l1} Lang, W. E. \emph{Extremal rational elliptic surfaces in characteristic $p$}.  I.  Beauville surfaces.  Math. Z. \textbf{207} (1991), 429-437.

\bibitem[L2]{l2} Lang, W. E. \emph{Extremal rational elliptic surfaces in characteristic $p$.  II}. Surfaces with three or fewer singular fibres.  Ark. Mat. \textbf{32} (1994), 423--448.

\bibitem[MP]{mp} Miranda, R.; Persson, U. \emph{On extremal rational elliptic surfaces}.  Math. Z. \textbf{193} (1986), 537-558.

\bibitem[Sil]{sil} Silverman, Joseph H. \emph{The Arithmetic of Elliptic Curves}, Graduate Texts in Math. \textbf{106}, Berlin-Heidelberg-New York, Springer-Verlag 1975.

\bibitem[Ta]{ta} Tate, J. \emph{Algorithm for determining the type of a singular fibre in an elliptic pencil, in Modular functions of one variable IV} (Birch, B. J., Kuyk, W., eds.), Graduate Texts in Math. \textbf{476}, Berlin-Heidelberg-New York, Springer-Verlag 1975.
 \end{thebibliography}
\end{document}